\documentclass[12pt]{amsart}
\usepackage{epsfig,color}

\makeatother

\theoremstyle{definition}

\def\fnum{equation} 
\newtheorem{Thm}[\fnum]{Theorem}
\newtheorem{Cor}[\fnum]{Corollary}

\newtheorem{Lem}[\fnum]{Lemma}

\newtheorem{Pro}[\fnum]{Proposition}

\numberwithin{equation}{section}

\newcommand{\nn}{{\bf{n}}}
\newcommand{\Ric}{{\text{Ric}}}

\newcommand{\Hess}{{\text {Hess}}}

\def\RR{{\bold R}}

\newcommand{\dv}{{\text {div}}}
\newcommand{\e}{{\text {e}}}

\newcommand{\cC}{{\mathcal{C}}}

\newcommand{\cI}{{\mathcal{I}}}

\newcommand{\cL}{{\mathcal{L}}}

\newcommand{\cH}{{\mathcal{H}}}

\newcommand{\cP}{{\mathcal{P}}}

\newcommand{\cR}{{\mathcal{R}}}

\newcommand{\eqr}[1]{(\ref{#1})}

\title{Optimal growth bounds for eigenfunctions}

\author{Tobias Holck Colding}%
\address{MIT, Dept. of Math.\\
77 Massachusetts Avenue, Cambridge, MA 02139-4307.}
\author{William P. Minicozzi II}%

\thanks{The  authors
were partially supported by NSF  DMS Grants 1812142, 2104349,  1707270 and 2005345.}

\email{colding@math.mit.edu and minicozz@math.mit.edu}

\begin{document}

\maketitle

\begin{abstract}
Analysis of non-compact manifolds almost always requires some controlled behavior at infinity.     Without such, one neither can show, nor expect, strong properties.   On the other hand, such assumptions restrict the possible applications and often too severely.    

In a wide range of areas non-compact spaces come with a Gaussian weight and a drift Laplacian.  Eigenfunctions are $L^2$ in the weighted space allowing for extremely rapid growth.  Rapid growth would be disastrous for many applications.    
Surprisingly, for very general tensors, manifolds and weights, we show the same polynomial growth bounds that Laplace and Hermite observed for functions on Euclidean space for the standard Gaussian.  This covers all shrinkers for Ricci and mean curvature flows.    

These results open a door for understanding general non-compact spaces.  It provides an analytic framework for doing nonlinear PDE on Gaussian
spaces where previously the Gaussian weight allowed wild growth that
made it impossible to approximate nonlinear by linear.     It is key to bound the growth of diffeomorphisms of non-compact manifolds and is the key for solving the ``gauge problem''.  The relative nature of the estimates and the slow growth in the bounds lead  to ``propagation of almost splitting'' that is significantly stronger than pseudo locality and key for applications.   
 \end{abstract}
 
 
 \section{Introduction}
 Laplace discovered that on the line eigenfunctions of $\cL\,u=u''-\frac{x}{2}\,u'$ in the Gaussian $L^2$ space are polynomials whose degree is exactly twice the eigenvalue.   These polynomials were later rediscovered twice.  First by Chebyshev and a few years later by Hermite.   They are now known as the Hermite polynomials and the eigenvalue equation as the Hermite equation.  They play an important role in diverse fields.  Here we prove a vast generalization of these results that has many applications.  
  
 We will prove optimal polynomial growth bounds for eigenfunctions of drift Laplacians in a general setting that includes
 all shrinking solitons for both Ricci and mean curvature flows (or MCF).  These bounds are sharp for the Ornstein-Uhlenbeck
 operator on Euclidean space.    
 Drift Laplacians are ubiquitous in many areas, including quantum field theory, stochastic PDE and anywhere the heat equation or Gaussian appear such as functional inequalities, parabolic PDEs, geometric flows, and probability.
 The drift term arises whenever there is a natural scaling or, more generally, a gradient flow.   
 
 The growth estimates that we prove hold
   in remarkable generality and without any assumptions on asymptotic decay.   This is surprising and in contrast to most other situations, like unique continuation, that require very strong geometric assumptions on the space.   A typical starting point for growth estimates is a Pohozaev identity or commutator estimate that come from a dilation, or approximate dilation, structure.   We have none of these here in this general setting.   In contrast, we rely on a miraculous cancellation for just the right quantity (see Section \ref{s:3}).
   A consequence of the generality is that the growth estimates hold  for all singularities which is key for applications.

 Two of the applications in \cite{CM2} are to solve the ``gauge problem'' and to show ``propagation of almost splitting''.  The ``gauge problem'' is a nonlinear problem that we solve using approximation  by a linear PDE and an iteration scheme.  This requires very strong growth estimates for the linear equation which the results here give.  A second application is to show propagation of almost splitting.  
The relative nature of the estimates and the slow growth in the bounds lead  to propagation of almost splitting that is significantly stronger than pseudo locality and will be key for applications.  One of the central tools for flows is pseudo-locality, \cite{Br}, \cite{P}, which roughly says that flatness propagates forward in time; accordingly, flatness propagates outward in space for shrinking solitons.  This has two important limitations: It gives little control over the metric itself because of the ``gauge group'' of diffeomorphisms and, second, there is a definite loss in the estimates that makes it impossible to iterate.
   In contrast, the growth bounds here
 will  show that metric control  itself propagates outwards, giving vastly more control.   Both the relative nature of the bounds and the metric control  play a key role.

Another application of our results is that polynomially growing ``special functions''  are dense in $L^2$. This  gives   manifold versions of some very classical problems in analysis.
Whereas Weierstrass's approximation theorem shows  that polynomials are dense among continuous functions on any compact interval, the classical Bernstein problem, \cite{Lu}, dating back to 1924, asks if polynomials are dense on $\RR$ in the weighted $L^p(\e^{-f}\,dx)$ space if $f$ is assumed to grow sufficiently fast at infinity.    On the line, the Hermite polynomials are dense in $L^2(\e^{-\frac{|x|^2}{4}}\,dx)$ and Lennart Carleson (and implicitly Izumi-Kawata) showed that polynomials are dense in $L^p(\e^{-|x|^{\alpha}}\,dx)$ if and only if $\alpha\geq 1$.    A similar problem in several complex variables is the {\it{completeness problem}}, going back to Carleman in 1923, about density of polynomials in weighted $L^2$ spaces of holomorphic functions; \cite{BFW}.

\subsection{Results} 
   In many settings one has an $n$-dimensional Riemannian manifold $(M,g)$, that could even be flat Euclidean space,  with two nonnegative functions $f$ and $S$ that satisfy
\begin{align}
\Delta\,f+S&=\frac{n}{2}\, ,\label{e:aronson1}\\
|\nabla f|^2+S&=f\, ,\label{e:aronson2}
\end{align}
and where $f$ is proper and $C^n$.    The weight $\e^{-f}$ gives a drift Laplacian $\cL$ on tensors $u$
\begin{align}
	\cL \, u = \e^f \, \dv \, \left( \e^{-f} \, \nabla u \right) =\Delta\, u -  \nabla_{\nabla f}\,  u 
\end{align}
that is self-adjoint with respect to the  $L^2$ norm $\| u \|_{L^2}^2 = \int |u|^2 \, \e^{-f}$.   
Since $|\nabla \sqrt{f}| \leq \frac{1}{2}$ by  \eqr{e:aronson2}, $b = 2\,\sqrt{f}$
satisfies  $|\nabla b|\leq 1$ as in \cite{CaZh}, cf. \cite{CxZh1}.
On $\RR^n$, $f=\frac{|x|^2}{4}$ and $S=0$ satisfy \eqr{e:aronson1}, \eqr{e:aronson2} with   $\cL=\Delta -\frac{1}{2}\,\nabla_x$  the Ornstein-Uhlenbeck operator and $b = |x|$.
   In Ricci flow, singularities are gradient shrinking solitons, $f$ is the potential  and $S$ is scalar curvature\footnote{See \cite{H}, \cite{Cn},  \cite{B}, \cite{Ca}, \cite{ChL},  \cite{ChLN}, \cite{CRF},  \cite{KL}, \cite{P}, \cite{T}}.  In MCF, singularities are shrinkers $\Sigma\subset \RR^N$, $f=\frac{|x|^2}{4}$ and $S=|{{\bf{H}}}|^2$, where ${\bf{H}}$ is the mean curvature vector\footnote{See, e.g., \cite{Hu}, \cite{CM1}, \cite{CM3}}.

 Throughout, $\lambda > 0$ is a constant and $u$ is a tensor on $M$.  We will often assume that
 \begin{align}
 	\langle \cL \, u  , u \rangle  \geq  - \lambda \, |u|^2 \, ;    \label{e:herelambda}
\end{align}
 this includes eigentensors with $\cL \, u = - \lambda \, u$.  
 To understand the growth of $u$,  we will study a weighted average of $|u|^2$ on level sets of  $b$
  \begin{align}
 	I(r) &= r^{1-n} \, \int_{b = r} |u|^2 \, |\nabla b| \, . \label{e:Iofr1} 
\end{align}
This is  defined at regular values of $b$, but   extends continuously to all values to be differentiable a.e.~and absolutely continuous.  The  weight $|\nabla b|$ will play a crucial role (cf. \cite{CM5}--\cite{CM7}, \cite{C}, \cite{AFM}, \cite{BS}, \cite{AMO}, \cite{GV}).   The  growth of $I$   will be bounded above in terms of  the solid integral
 \begin{align}
 	D(r) &= r^{2-n} \, \e^{ \frac{r^2}{4}} \, \int_{b < r} \left( |\nabla u|^2 + \langle \cL \, u , u \rangle	\right)  \, \e^{-f} \, .
	\label{e:integralformD}
 \end{align}
 The frequency $U = \frac{D}{I}$ is defined when $I$ is positive and will measure the    growth of $\log I$.

 The next theorem shows that an $L^2$ tensor satisfying \eqr{e:herelambda}
  has frequency bounded by $2\, \lambda$ and, accordingly, it grows at most polynomially at this rate.  This may seem surprising since the weight $\e^{-f}$ decays rapidly, so the $L^2$ condition a priori allows extremely rapid growth.   
 
 \begin{Thm}   \label{t:main}
 Suppose $u , \cL \, u \in L^2$, \eqr{e:aronson1}, \eqr{e:aronson2}, \eqr{e:herelambda} hold, and $u$ does not vanish identically outside a compact set.
 Given $\epsilon > 0$, there exists $R=R(n,\lambda,\epsilon)$ so if $r>R$, then
\begin{align}
U(r)\leq 2\,\lambda\,\left(1+\frac{4\,\lambda+2\,n-4+\epsilon}{r^2}\right)\, 	\label{e:0p9}
\end{align}
 and for all $r_2>r_1>R$ and $c= 2\lambda \,\left(4\,\lambda+2\,n-4+\epsilon\right)$
 \begin{align}
 I(r_2)\leq I(r_1)\,\left(\frac{r_2}{r_1}\right)^{4\,\lambda} \, \e^{c\, \left( r_1^{-2} -r^{-2}_2 \right)}\, .	\label{e:IU08}
 \end{align}
 \end{Thm}
 
 This is sharp for the Ornstein-Uhlenbeck operator on $\RR^n$  where the $L^2$ eigenfunctions are Hermite polynomials with degree twice the eigenvalue.  In fact, the upper bound
  \eqr{e:0p9} is sharp not just in the $2\, \lambda$ in front, but in all the other constants as  well as can be seen from the Hermite polynomials.
The  $R$ in Theorem \ref{t:main} does not depend  on $f$, $M$ or $S$. The  theorem still holds if \eqr{e:aronson1}, \eqr{e:aronson2}, and 
 \eqr{e:herelambda} hold outside of a compact set.  Moreover, it holds with obvious changes when the constant $n$ in \eqr{e:aronson1} is replaced by any other constant.    
 Finally, note that $u$ cannot vanish on an open set if $u$ has unique continuation, e.g. if $\cL \, u = - \lambda \, u$ by \cite{Ar}.

One application will be to  gradient shrinking Ricci solitons.
  A gradient shrinking Ricci soliton $(M,g,f)$ is a Riemannian manifold $(M,g)$ and function $f$ satisfying
$ 	\Ric + \Hess_f = \frac{1}{2} \, g $.   
  The standard drift Bochner formula gives that if $\cL \, v = -\left( \frac{1}{2} + \lambda \right) \, v$, then $\cL \, \nabla v = - \lambda \, \nabla v$ and
  \eqr{e:herelambda} applies to $u = \nabla v$:

\begin{Cor}	\label{c:gsrs}
If $(M,g,f)$ is a gradient shrinking soliton, then \eqr{e:0p9} and \eqr{e:IU08} hold if  $u = \nabla v$ where $v$ is an eigenfunction with eigenvalue $\lambda + \frac{1}{2}$.
\end{Cor}

An important application is to vector fields on a shrinking Ricci soliton, where it is used  in \cite{CM2} to show propagation of almost splitting.   
  If a shrinker almost splits on a scale, then it has an eigenvalue close to $\frac{1}{2}$.  
If an eigenfunction has eigenvalue close to $\frac{1}{2}$, then its gradient will have eigenvalue close to $0$.  It will then follow from
Corollary \ref{c:gsrs} that the gradient is nearly constant on a much larger scale than one would expect.
This will be  key in the  propagation of almost splitting  in \cite{CM2}.

 For some applications, it will be useful to consider a more general case  where $u$ satisfies  
\begin{align}
	\langle \cL \, u , u \rangle  \geq - \lambda \, |u|^2 - \psi  \, ,  \label{e:poisson}
\end{align}
where   $\psi$ is a nonnegative function.  
 Define the quantity $J$ by
 \begin{align}    \label{e:J}
 	J(r) = \int_{b < r} b^{2-n} \, \psi  \, .
 \end{align}
  The next theorem gives  polynomial growth in terms of $\lambda$ and $J$.

 \begin{Thm}   \label{t:main2A}
If $u , \cL \, u \in L^2$,  \eqr{e:aronson1}, \eqr{e:aronson2}, \eqr{e:poisson} hold, $\delta\in (0,2)$ and $r_2 > r_1 \geq R(\lambda,n,\delta)$,
then
 \begin{align}
	  I(r_2) \leq \left( \frac{r_2}{r_1} \right)^{ 4 \, \lambda +  \delta} \, \left\{ I(r_1) +  \frac{20 \, \sup J}{4 \, \lambda + \delta} \right\} \, .		\label{e:main2A}
\end{align}
  \end{Thm}

  \vskip1mm
  Theorem \ref{t:main2A} is used in \cite{CM2} to solve the ``gauge problem'' on a non-compact manifold.  Namely, to solve the ``gauge problem'' we will solve a nonlinear system of PDEs and prove optimal bounds for solutions.  The PDE produces a diffeomorphism  that  fixes an appropriate gauge in the spirit of the slice theorem for group actions.  We then show optimal bounds for the displacement function of the diffeomorphism. 
To do this, we first infinitesimally bound the growth of a diffeomorphism on a Ricci shrinker in \cite{CM2}.  This is done using Theorem \ref{t:main2A} to bound the growth of solutions to a Poisson-type equation $\cP\,Y=V$.  Here $V$ is a known vector field, $Y$ is unknown and $\cP$ a complicated system operator that is the linearization of the nonlinear system of PDEs that fixes the gauge.    Even though the operators $\cP$ and $\cL$ are very different, we will show in \cite{CM2} that one can deduce growth bounds for solutions of the Poisson equation for $\cP$ from growth bounds for the Poisson eigenvalue equation for $\cL$.

\vskip1mm
There is a long history of studying the growth of solutions to differential equations, inequalities, and systems.  At a very rough level, there are two main techniques.  The first, exemplified in the work of Carleman and H\"ormander, is to consider weighted $L^2$ norms with growing weights.  The second, seen for instance in the work of Hadamard and Almgren, is to study the growth of spherical maxima or averages.  The second is an extreme version of the first where the weight is a measure concentrated on a lower dimensional set.
As such, the second method typically gives stronger information and requires  greater structure, such as invariance under dilations.
  Almgren's frequency has  been used to show unique continuation, \cite{GL}, and structure of the nodal sets, \cite{Lo}; prior to this, the main tool in unique continuation was Carleman estimates that still is the primary technique.
Almgren's frequency bounds relied on scaling for $\RR^n$; cf. \cite{CM5}, \cite{CM6}.  The papers \cite{Be} (cf. \cite{W}), \cite{CM4} developed frequencies   for conical and cylindrical MCF shrinkers and did not involve a weight like $|\nabla b|$.
Theorems \ref{t:main}, \ref{t:main2A}, in contrast, hold very generally, including for all shrinkers in both Ricci flow and MCF.
A much weaker version of Theorem \ref{t:main}, that was not relative, was proven in \cite{CM3} in the special case of MCF.  

\section{The level sets of $b$ and the properties of $I$ and $D$}

We will define   $D(r)$ and $I(r)$  as solid integrals over sub-level sets $\{ b < r\}$ of a proper $C^n$ function $b$.
For these functions to be continuous, we must show that level sets of $b$ have measure zero. This is (2) in the next lemma; (1) will be used to prove absolute continuity,  while (3) will be used to show that $I>0$.  Since $b$ is $C^n$, Sard's theorem gives that almost every level set is regular.

\begin{Lem}   \label{l:preparation}
Suppose that $f:M\to \RR$ is a proper function with $\cL\,f=\frac{n}{2}-f$.  Let $\cC$ denote the set of critical points of $f$
and $\cH_r$ the boundary of $\{f>\frac{r^2}{4}\}$.  We get for $r>\sqrt{2\,n}$ that:
\begin{enumerate}
\item The critical set $\cC$ in $\{ f> \frac{n}{2} \}$  is locally contained in  a smooth $(n-1)$-manifold.
\item Each level set $\{ f = c \}$ for $c \geq \frac{n}{2}$ has $\cH^n ( \{ f = c \}) = 0$.
\item The regular set $\cR_r=\cH_r\setminus\cC$ is dense in $\cH_r$.     
\end{enumerate}
\end{Lem}

\vskip1mm
The nodal sets of eigenfunctions have a great deal of structure, but the value zero is special and many properties do not hold for non-zero values.  In fact, it is possible  to have a level set that is entirely critical, as  occurs at the local extrema for the radial eigenfunction $J_0 (|x|)$ on $\RR^2$ where $J_0$ is the Bessel function of the first kind.   However, by (3), this does not  occur for the subset $\cH_r$ of $\{ f= r\}$ that is the boundary of $\{ f > r \}$.

\begin{proof}[Proof of Lemma \ref{l:preparation}]
Note first that  $\cL\,f<0$ on $\{f>\frac{n}{2}\}$ and, thus,  $\Delta\,f<0$ on $\cC\cap \{f>\frac{n}{2}\}$.    Working in a neighborhood of a critical point we can therefore choose a coordinate system $\{x_i\}$ so that  $\partial_{x_1}^2f<-1$.  If $x\in \cC$,  then $\partial_{x_1}f(x)=0$ and thus by the implicit function theorem we can choose a new coordinate system in a neighborhood of $x$ so that in those coordinates $\{\partial_{x_1}f=0\}\subset \{y_1=0\}$ and so that $\partial_{x_1}$ is transverse to $\{y_1=0\}$.  We therefore have that (nearby) $\cC\subset \{\partial_{x_1}f=0\}\subset \{y_1=0\}$.     This gives (1). 

 For $c > \frac{n}{2}$, claim (2) follows from (1) since $\{ f = c \} \setminus \cC$ is a countable union of $(n-1)$-manifolds.  The borderline case $c = \frac{n}{2}$ in (2) follows from \cite{HHL}.
 
 We turn next to (3).  Note first that at $x=(x_1,\cdots,x_n)\in \cC$ if we let $h(s)=f(s,x_2,\cdots,x_n)$, then $h'(x_1)=0$ and $h''(x_1)<0$ so $h$ has a strict local maximum at $x_1$.  In particular, any neighborhood of any $x\in \cC\cap \{f>\frac{n}{2}\}$ intersects $\{f<f(x)\}$.  
Suppose now  that the conclusion (3) fails; so suppose that there exists $x\in \cH_r$ and a neighborhood $O$ so that $O\cap \cH_r\subset \cC$.  It follows that $O\cap \cH_r\subset \{y_1=0\}$.  Since $O\cap \cH_r$ separates the two non-empty sets $O\cap \{f>\frac{r^2}{4}\}$ and $O\cap \{\frac{r^2}{4}>f\}$ and $O\cap \cH_r$ is contained in $\{y_1=0\}$ it follows that $O\cap \cH_r=O\cap \{y_1=0\}$ and after possibly changing the orientation of $y_1$ we may assume that $O\cap \{y_1>0\}\subset \{f>\frac{r^2}{4}\}$ and $O\cap \{y_1<0\}\subset \{f<\frac{r^2}{4}\}$.     This, however, contradicts that at $x$ we have that $\partial_{x_1}^2f<0$ and $\partial_{x_1}$ is transverse to the level set $\{y_1=0\}$ so both $O\cap \{y_1>0\}$ and $O\cap \{y_1<0\}$ contains points where $f<f(x)=\frac{r^2}{4}$.  
\end{proof}

The functions $I(r)$, $D(r)$ and $U(r)$ may not be differentiable everywhere, but they will be absolutely continuous and differentiable a.e.
A function $Q(r)$ is  {\it{absolutely continuous}} on an interval $\cI$ if for every $\epsilon > 0$, there exists $\delta > 0$ so that if $\cup_{\alpha} (r_{\alpha} , R_{\alpha})$ is a finite disjoint union of intervals  in $\cI$ with $\sum (R_{\alpha} - r_{\alpha}) < \delta$, 
then we have $\sum \left| Q(R_{\alpha}) - Q(r_{\alpha}) \right| < \epsilon$.  
Absolutely continuous functions are precisely the ones where the  fundamental theorem of calculus holds (\cite{F}, page $165$): $Q$ is absolutely continuous if and only if 
  it is continuous, differentiable a.e., the derivative is in $L^1$, and  for every $r_1 < r_2$
\begin{align}	\label{e:AC}
	Q(r_2) - Q(r_1) = \int_{r_1}^{r_2} Q'(t) \, dt \, .
\end{align}
We will use   the following standard fact:
If $Q_1$ and $Q_2$ are absolutely continuous and $W:\RR^2 \to \RR$ is Lipschitz on the range of $(Q_1 , Q_2)$, then $W(Q_1 , Q_2)$ is  absolutely continuous.

\begin{Lem}	\label{l:coareaapp}
Suppose that $b$ is a proper $C^n$ function and $\cH^n (|\nabla b| = 0) = 0$ in $\{ b \geq r_0\}$ for some fixed $r_0$.  If $g$ is a bounded function and $Q(r) = \int_{r_0 < b < r} g$, 
then $Q$ is absolutely continuous and  $Q'(r) = \int_{b =r} \frac{g}{|\nabla b|}$ a.e.
\end{Lem}

\begin{proof}
By separately considering the positive and negative parts of $g$, it suffices to assume that $g\geq 0$ is bounded.
Define a sequence of functions $Q_i$ by
\begin{align}
	Q_i (r) = \int_{r_0 < b < r}  \frac{g\, |\nabla b|}{|\nabla b| + i^{-1}} \, .
\end{align}
The functions $ \frac{g\, |\nabla b|}{|\nabla b| + i^{-1}}$ are bounded above by $g$ everywhere and converge to the bounded function $g$ a.e. (since  $\cH^n (|\nabla b| = 0) = 0$), so $\lim_{i \to \infty} \, Q_i (r) = Q(r)$ by
the dominated convergence theorem. 
Define functions $q_i(t)$ and $q(t)$ at regular values $t$ of $b$ by
\begin{align}
	q_i (t) = \int_{b=t}  \frac{g}{|\nabla b| + i^{-1}}  {\text{ and }}
	q(t)  = \int_{b=t}  \frac{g}{|\nabla b|} \, .
\end{align}
Since $b$ is $C^n$, Sard's theorem ($3.4.3$ in \cite{F}) gives that a.e.~$t$ is a regular value of $b$ and, thus, these
 functions are defined a.e.
The co-area formula (\cite{F}, page $243$) gives that
\begin{align}	\label{e:fromco}
	Q_i(r) = \int_0^r \, q_i (t) \, dt \, .
\end{align}
The sequence $q_i$ is monotonically increasing with $q_i \leq q_{i+1} \leq \dots < q$.
Moreover, $q_i$ converges to $q$ a.e.  The monotone convergence theorem gives that
\begin{align}
	\lim_{i\to \infty} \, \int_0^r \, q_i (t) \, dt = \int_0^r \, q(t) \, dt \, .
\end{align}
Combining this with \eqr{e:fromco} and $\lim_{i \to \infty} \, Q_i (r) = Q(r)$  gives the lemma.
\end{proof}

\subsection{Absolute continuity of $I$ and $D$}

In the remainder of this paper, we specialize to $f$ satisfying \eqr{e:aronson1} and \eqr{e:aronson2} and $b = 2\, \sqrt{f}$.  It follows that
   \begin{align}	\label{e:rho1}
 	|\nabla b|^2 &= 1 - \frac{4\, S}{b^2} \leq 1 \, , \\
	b \, \Delta\, b &=  n - |\nabla b|^2 - 2 \, S 
	\, .  \label{e:rho2}
 \end{align}
 Since $f$ is nonnegative and proper, then so is $b$ and, thus, the level sets of $b$ are compact.    Furthermore, Lemma 
 \ref{l:preparation} applies and, thus, so does Lemma \ref{l:coareaapp}.
 
  The definition \eqr{e:Iofr1} of $I(r)$ at regular values of $b$ will be extended continuously  to all values next.  To do this, choose a regular value $r_0 < 2\, \sqrt{2n}$ of $b$ and set
  \begin{align}
 	 I(r) 	 &=
 \int_{r_0 < b < r} b^{1-n} \, \left\{  \, \langle \nabla |u|^2 , \nabla b \rangle + \frac{|u|^2}{b^3} \, 2\, S \left( 2\,n - b^2	\right)
	\right\}
	+ \int_{b= r_0} |u|^2 \, |\nabla b|  \, . \label{e:Iofr1A}
 \end{align}
 The reason for stopping the integral at $b=r_0$ is that $b^{1-n}$ and $S\, b^{-2-n}$ might not be integrable in the interior if $\min b = 0$.

 \begin{Lem}	\label{l:IBDR}
  At regular values $r$ of $b$, 
the definitions \eqr{e:Iofr1} and \eqr{e:Iofr1A}  of $I(r)$ agree and 
   \begin{align}
   	D(r) &= \frac{r^{2-n}}{2} \, \int_{b = r} \langle \nabla |u|^2 , \frac{\nabla b}{|\nabla b|} \rangle \, . \label{e:firstclaim}
\end{align}
  \end{Lem}
  
  \begin{proof}
  To see that \eqr{e:Iofr1} and \eqr{e:Iofr1A}  agree at regular values, observe that
 the unit normal to the level set $b = r$ is given, at regular points, by $\nn = \frac{\nabla b}{|\nabla b|}$, so we can rewrite \eqr{e:Iofr1}
  \begin{align}
 	  r^{1-n} \, \int_{b=r} |u|^2 |\nabla b| &=  \int_{b < r} \dv \, (|u|^2 \, b^{1-n} \, \nabla b) \notag \\
	 &=
	\int_{b < r} b^{1-n} \, \left\{  \, \langle \nabla |u|^2 , \nabla b \rangle + |u|^2 \, \left( \Delta \, b - \frac{(n-1) \, |\nabla b|^2}{b}
	\right)
	\right\} \, .   \label {e:Iofr2} 
 \end{align}
  By \eqr{e:rho1} and \eqr{e:rho2}, we have that
 $b \, \Delta\, b =  n - |\nabla b|^2 - 2 \, S$ and $|\nabla b|^2 = 1 - \frac{4\, S}{b^2} $ and, thus,
 \begin{align}
 	b\, \left( \Delta \, b - \frac{(n-1) \, |\nabla b|^2}{b} \right) & = n\,(1-|\nabla b|^2) - 2 \, S 
	= \frac{2\, S}{b^2} \, (2\,n - b^2) \, .
 \end{align}
 Substituting this into \eqr{e:Iofr2} gives \eqr{e:Iofr1A}.
 The divergence theorem gives
  \begin{align}
\int_{b = r} \langle \nabla |u|^2 , \frac{\nabla b}{|\nabla b|} \rangle =  \e^{ \frac{r^2}{4}} \, \int_{b < r}
	\dv \, \left( \nabla |u|^2 \, \e^{-f} 
	\right) =  \e^{ \frac{r^2}{4}} \, \int_{b < r}
	   \cL \, |u|^2	   \, \e^{-f} 
	\, .
 \end{align}
    Multiplying this by $ 	\frac{r^{2-n}}{2} $ gives \eqr{e:firstclaim}.  
  \end{proof}

 \begin{Lem}	\label{l:Idiff}
  Both $I(r)$ and $D(r)$ are absolutely continuous with derivatives given a.e. by
   \begin{align}
 	 	 I'(r) &= 	r^{1-n} \, \int_{b = r}  \, \langle \nabla |u|^2 , \frac{\nabla b}{|\nabla b|} \rangle + \left( 2\,n\,r^{-2} - 1	\right)\,r^{1-n}\int_{b=r}\frac{2\, S \, |u|^2}{r\, |\nabla b|} \, ,  \label{e:Iofr3} \\
		 D'(r) &= \frac{2-n}{r} \, D + \frac{r}{2} \, D + \frac{r^{2-n}}{2} \, \int_{b = r}
	\frac{  \cL \, |u|^2 }{|\nabla b|} \, . \label{e:Dpri}
\end{align}
Where $I$ is positive $\log I$ is absolutely continuous and the derivative is given a.e. by
\begin{align}
	r\, (\log I)'(r) = 2\, U + (2\,n\,r^{-2}-1)\,\frac{2\, r^{1-n}}{I}  \int_{b = r} \frac{S\,|u|^2}{ |\nabla b|}  \, . \label{e:secondclaim}
\end{align}
  Furthermore,   $(\log I)' \leq 2 \, U/r$ a.e. when
 $r \geq \sqrt{2\,n}$. 
  \end{Lem}

 \begin{proof}
  The continuity of $I(r)$ (as defined in  \eqr{e:Iofr1A}) and $D(r)$ follows from the dominated convergence theorem since $ \cH^n \left(\{ b= r \} \right) = 0 $ by
 (2) in  Lemma \ref{l:preparation}.  Furthermore, Lemma \ref{l:coareaapp} applies to both $I$ and $D$ and, thus, both are absolutely continuous   and 
  $I
'$ is given a.e. by  \eqr{e:Iofr3} and $D'$ is given a.e. by \eqr{e:Dpri}.  
Equation \eqr{e:secondclaim} follows from \eqr{e:firstclaim} and  \eqr{e:Iofr3}.  
  Since $S\geq 0$, we see that $ \left(\log I\right)'=\frac{I'}{I}\leq \frac{2\,U}{r}$
 for $r\geq \sqrt{2\,n}$.  
 \end{proof}

 \section{Positivity of $I(r)$}	\label{s:Ipo}

 The main result of this section is that $I(r) > 0$ when $r$ is sufficiently large:
 
  \begin{Pro}	\label{p:Ipositive}
 If $ u , \cL u  \in L^2$ and \eqr{e:herelambda} holds, then either
 \begin{itemize}
 \item[(A)] $I(r) > 0$ for every $r > 2 \, \sqrt{n+ 4\, \lambda}$, or
 \item[(B)] $u$ vanishes identically outside of a compact set.
 \end{itemize}
  \end{Pro}
  
  An immediate consequence of (A) in Proposition \ref{p:Ipositive} is that 
 $U(r)$ is well-defined  and absolutely continuous for  $r > 2 \, \sqrt{n+ 4\, \lambda}$, and $U'$ is given a.e. by
 \begin{align}
 	U'(r) = \frac{D'}{I} - \frac{D \, I'}{I^2} \, .
 \end{align}

 The next elementary lemma shows that $|u| \in W^{1,2}$ and $|u| \, |\nabla f| \in L^2$ if $u , \cL \, u \in L^2$ (cf. \cite{CxZh2}, \cite{CM3}).

\begin{Lem}	\label{l:W12}
 If $ u , \, \cL \, u \in L^2$, then $|\nabla |u|| $, $|\nabla u|$, $|u| \, \sqrt{f}$, and  $|u| \, |\nabla f|$ are all in $L^2$.
\end{Lem}

\begin{proof}
By the Kato inequality  and \eqr{e:aronson2},   $|\nabla |u|| \leq |\nabla u|$ and  $|\nabla f|^2 \leq f$. Thus,  it suffices to prove that $|\nabla u| , |u| \, \sqrt{f} \in L^2$.
We show first that $|\nabla u| \in L^2$.
Let $\eta$ be   a compactly supported function with $| \eta | , |\nabla \eta| \leq 1$.  Since
	$\cL \, |u|^2 = 2 \, |\nabla u|^2 +2 \, \langle u , \cL \, u \rangle$, 
applying the divergence theorem to $\eta^2 \, \nabla |u|^2 \, \e^{-f}$ gives 
\begin{align}
	 \int \eta^2 \, |\nabla u|^2 \,  \e^{-f} \leq  \| u \|_{L^2} \, \| \cL \, u \|_{L^2} +2 \, \int |\eta| \, |\nabla \eta|  \, |u| \, |\nabla |u|| \, \e^{-f} \, .
\end{align}
Using $|\nabla |u|| \leq |\nabla u|$ and the absorbing inequality 
$2 \, |\eta| \, |u| \, |\nabla u| \leq 2 \,  |u|^2 + \frac{1}{2} \, \eta^2 \, |\nabla u|^2$, we can absorb the $|\nabla |u||$ term and then apply the monotone convergence theorem for a sequence of $\eta$'s going to one everywhere gives that $|\nabla u| \in L^2$.
To see that $|u| \, \sqrt{f} \in L^2$, apply the divergence theorem to $\eta^2 \, |u|^2 \, \nabla f \, \e^{-f}$ and use that $\cL \, f = \frac{n}{2} - f$ to get
\begin{align}
	\int \eta^2 \, |u|^2 \, \left( f - \frac{n}{2} \right) \, \e^{-f} \leq 2 \, 
	\int \left\{ \eta^2 \, |u| \, |\nabla |u|| \, |\nabla f| + |\eta| \, |\nabla \eta| \, |u|^2 \, |\nabla f|
	\right\} \, \e^{-f} \, .
\end{align}
Using the bound $|\nabla f|^2 \leq f$, we can use absorbing inequalities on both terms on the right and then use that $|u|, |\nabla |u||$ are in $L^2$ to conclude that $|u| \, \sqrt{f} \in L^2$.
\end{proof}

    We will need a few preliminary results, including the following consequence of 
  Lemma \ref{l:preparation}:
 
\begin{Cor}     \label{c:preparation}
If $I(r)=0$ and \eqr{e:herelambda} holds, then $u\equiv 0$ on $\cH_r$.    
\end{Cor}

\begin{proof}
Suppose $x\in \cH_r$ with $|u|(x)>0$.  Since $u$ is continuous it follows from Lemma \ref{l:preparation} that there exists another point $y\in \cH_r\setminus \cC$ where $|u|(y)>0$.   Since $y$ is a regular point, then in a neighborhood of $y$ we have that $|u|\geq \frac{|u|(y)}{2}>0$, $|\nabla b|\geq \frac{|\nabla_y b|}{2}>0$.  It follows that there exists an $\nu>0$ such that  if $s$ be any regular value  sufficiently close to $r$, then the level set $b=s$ is a smooth hyper-surface and $I(s)\geq \nu>0$.    The claim follows.    
\end{proof}

 \begin{proof}[Proof of Proposition \ref{p:Ipositive}]
 Suppose that (A) fails and, thus, $I(r) = 0$ for some $r > 2 \, \sqrt{n+4\, \lambda}$.    
 By Corollary \ref{c:preparation},  we know that $|u| =0$ on $\cH_r = \partial \{ b > r \}$.      Assume (B) also fails and
 choose a connected component $\Omega $ of $\{ |u| > 0 \}$ with
 	$\Omega \subset \{ b > r \}$.  
This will lead to a contradiction.

 By Lemma \ref{l:W12}, $|u|$, $|u| \, |\nabla f|$, $|\nabla u|$ and $|\nabla |u||$ are all in $L^2$.
 For each $j$, let 
 $\eta_j: \RR \to [0, \infty)$ be a smooth function with $0 \leq \eta_j' \leq 4$ and
 \begin{align}
 	\eta_j (x) = 
	\begin{cases}
	x  & {\text{ for }} \frac{1}{j} \leq x \, , \\
	0 & {\text{ for }} x \leq \frac{1}{2j} \, .
	\end{cases}
 \end{align}
 Let $\chi$ be the characteristic function of $\Omega$, i.e, $\chi$ is one on $\Omega$ and zero otherwise, and 
  define  
  	$v_j = \eta_j (|u|) \, \chi_{\Omega} $.
Note that each $v_j$ is smooth on all of $M$ and $v_j \in W^{1,2}$ since $v$ is and $\eta_j$ is Lipschitz. 
Moreover, $v_j$ has support in $\{ b \geq r \}$ since $\Omega \subset \{ b > r \}$.

Let $V$ be a vector field with $V \in L^2$ and $v_j \, \left( \dv \, V - \langle V , \nabla f \rangle \right) \in L^1$.  Given $\eta$ with compact support and  $|\eta| , \, |\nabla \eta| \leq 1$, applying the divergence theorem to $ \eta \, v_j \, V \, \e^{-f} $ gives
\begin{align}	 
	\int \eta \,  \left( \langle \nabla v_j , V \rangle + v_j \, \left( \dv \, V - \langle V , \nabla f \rangle \right) \right) \, \e^{-f}  = - \int v_j \, \langle V , \nabla \eta \rangle \, \e^{-f}   \, .
\end{align}
Taking a sequence of $\eta$'s converging to one, the dominated convergence theorem gives
\begin{align}	\label{e:willdom}
	\int \left( \langle \nabla v_j , V \rangle + v_j \, \left( \dv \, V - \langle V , \nabla f \rangle \right) \right) \, \e^{-f} = 0 \, .
\end{align}
By the Lipschitz bound on $\eta_j$ and the Kato inequality,
 $|v_j| \leq 4 \, |u| $ and  $ |\nabla v_j| \leq   4 \, |\nabla u|$.
Furthermore, $v_j \to |u| \, \chi$ and $\nabla v_j \to \chi \, \nabla |u|$ a.e.~(since $\nabla |u| = 0$ a.e. on  $\{ |u| = 0\}$  by, e.g., lemma $7.7$ in \cite{GiTr}). Thus, applying the dominated convergence theorem to \eqr{e:willdom} gives
\begin{align}	\label{e:nowdom}
	\int_{\Omega} \left( \langle \nabla |u| , V \rangle + |u|  \, \left( \dv \, V - \langle V , \nabla f \rangle \right) \right) \,  \e^{-f} = 0 
	\, .
\end{align}
 First, we apply this with $V = \nabla |u|$ and then use \eqr{e:herelambda} and $|\nabla |u|| \leq |\nabla u|$ to get
\begin{align}		\label{e:combiii}
	0 =     \int_{\Omega} \left( |\nabla u|^2 + \langle u , \cL \, u \rangle \right) \,  \e^{-f}
	\geq    \int_{\Omega} \left( |\nabla |u||^2 - \lambda \, |u|^2 \right) \,  \e^{-f}
	\, .	 
\end{align}
 For the second application of \eqr{e:nowdom}, take $V = |u| \, \nabla f$ and use $\cL \, f = \frac{n}{2} - f$ to get
 \begin{align}
 	0 =  \int_{\Omega} \left\{ 2\,  \langle |u| \, \nabla |u| , \nabla f \rangle 
	+|u|^2 \, \cL \, f
	  \right\} \,  \e^{-f} =   \int_{\Omega} \left\{ 2\,  \langle |u| \, \nabla |u| , \nabla f \rangle 
	+|u|^2 \, \left(  \frac{n}{2} - f \right)	  \right\} \,  \e^{-f} \, . \notag
 \end{align}
Since $|\nabla f|^2 \leq f$, the absorbing inequality
$2\, \left| \langle |u| \, \nabla |u| , \nabla f \rangle \right| \leq 2 \, |\nabla |u||^2 + \frac{1}{2} \, |u|^2 \, |\nabla f|^2$ gives
 \begin{align}
 	\int_{\Omega} |u|^2  \,  \left(  f - n \right)  \, \e^{-f} \leq 
	4\, \int_{\Omega}  |\nabla |u||^2   \, \e^{-f} \leq 4\, \lambda \, \int_{\Omega}  |u|^2   \, \e^{-f}\, ,
 \end{align}
 where the last inequality is  \eqr{e:combiii}.
 Since $|u| > 0$ and $f = \frac{b^2}{4} > \frac{r^2}{4}$ on $\Omega$, we get that
 	 	$\left(  \frac{r^2}{4} - n - 4 \, \lambda \right)    \leq 0$.
 This is the desired contradiction since $r > 2 \, \sqrt{n+ 4\, \lambda}$.
 \end{proof}

 \section{Growth estimates}	\label{s:3}
 
 Throughout this section, we assume that $u$ satisfies \eqr{e:herelambda}.  We will use that, by Lemma \ref{l:Idiff}, $\log I$, $\log D$, and $\log U$ 
 are absolutely continuous as long as $I, D > 0$.  One challenge for controlling the growth of $D$ and $I$ is that $D'$ and $I'$ 
 have  terms involving $S$, with the wrong sign in one case and  a variable sign in the other.  The terms will be played off each other and we will be able to control  the right combination; this miraculous cancelation makes it work.

\begin{Pro}	\label{p:driftUmont}
If $r$ is a regular value of $b$ and $D(r) , I(r) > 0$, then 
\begin{align}
 	r\,(\log D)' &\geq 2-n  + \frac{r^2}{2}  + U -\frac{\lambda\,r^2}{U} 
	 -\frac{4\,\lambda}{U\,I} \, r^{1-n}\, \int_{b = r}
	\frac{S\,|u|^2}{|\nabla b|} 
	\, , \label{e:thederivofD2}\\
	r\, \left(\log U\right)' (r)&\geq    2-n+\frac{r^2}{2}-U-\frac{\lambda\,r^2}{U} +
	\left(1-\frac{2\,\lambda}{U}- \frac{2\,n}{r^2}\right) \, \frac{2\,r^{1-n}}{I}\int_{b=r}\frac{S\,|u|^2}{|\nabla b|}
\, .\label{e:logU}
\end{align}
\end{Pro}
 
 \begin{proof}
 Lemma \ref{l:Idiff} and \eqr{e:herelambda}  give
 \begin{align}
	D'(r) &= \frac{2-n}{r} \, D + \frac{r}{2} \, D + \frac{r^{2-n}}{2} \, \int_{b = r}
	\frac{  \cL \, |u|^2 }{|\nabla b|}  \geq 
	 \frac{2-n}{r} \, D + \frac{r}{2} \, D + r^{2-n} \, \int_{b = r}
	\frac{ \left( |\nabla u|^2 - \lambda \, |u|^2	\right)}{|\nabla b|} 
	\, . \notag
 \end{align}
  Since $4\, S= b^2 - b^2 \, |\nabla b|^2$, we get that
    \begin{align}
	r\,  (\log D)'(r) &\geq 2-n +\frac{r^2}{2}-\frac{\lambda\,r^2}{U}+\frac{r^{3-n}}{D}\,\int_{b=r}\frac{|\nabla u|^2}{|\nabla b |}-\frac{4\,\lambda\,r^{1-n}}{D}\int_{b=r} \frac{S\,|u|^2}{|\nabla b|}\, .  \label{e:thirdclaim}
 \end{align}
 Note that by the Cauchy-Schwarz inequality
 \begin{align}
 D^2(r)&=\left(\frac{r^{2-n}}{2} \, \int_{b =r} \langle \nabla |u|^2 , \frac{\nabla b}{|\nabla b|} \rangle\right)^2\leq    I(r) \, r^{3-n} \, \int_{b = r} \frac{| \nabla u|^2}{ {|\nabla b|}} \, .
 \end{align} 
Dividing this by $I(r)$ gives $U\,D\leq r^{3-n}\int_{b=r}\frac{|\nabla u|^2}{|\nabla b|}$.
Using this in \eqr{e:thirdclaim} gives \eqr{e:thederivofD2}.
Combining  \eqr{e:thederivofD2}  and \eqr{e:secondclaim} gives \eqr{e:logU}.  
 \end{proof}
 
 An immediate consequence of the proposition is the following:

\begin{Cor}	\label{c:driftUmont}
If $r$ is a regular value with $U(r)>2\,\lambda$ and $r>\sqrt{\frac{2\,n}{1-\frac{2\,\lambda}{U}}}$, then
\begin{align}
	\left(\log U\right)' &\geq    \frac{2-n-U}{r} +    r\,\left(\frac{1}{2}- \frac{\lambda}{U}\right)
\, .
\end{align}
\end{Cor}

We use this to show that if $U$ goes strictly above $2\, \lambda$, then it grows quadratically; this does not assume that $u \in L^2$ and, indeed, it is impossible when $u , \cL \, u \in L^2$.
 
 \begin{Thm}	\label{c:Uquad}
 Given $\delta > 0$, there exists $R > \sqrt{2n}$ so that if 
  $U (r_0) > (2+\delta)\,\lambda$ for some $r_0 \geq R$, then $ U(r)\geq \frac{1}{2}\,r^2-r$ for every $r$ sufficiently large.
 \end{Thm}
 
 \begin{proof}
 If $U(r)> (2+\delta)\,\lambda$ for a regular value $r > \sqrt{ (4n\, (2+\delta)/\delta)}$, then Corollary \ref{c:driftUmont} gives
 \begin{align}
 (\log U)' (r)\geq \frac{2-n-U}{r}+\frac{\delta\,r}{2\,(2+\delta)}\, .
 \end{align}
 It follows that if $ (2+\delta)\,\lambda < U < \frac{\delta\,r^2}{5\,(2+\delta)}$ and $r > \sqrt{ (4n\, (2+\delta)/\delta)}$, then
\begin{align}
	 (\log U)' (r)> \frac{\delta\,r}{2\,(2+\delta)} - \frac{n}{r} - \frac{\delta\,r}{5\,(2+\delta)} > \frac{\delta \, r}{2+\delta} 
	 \, \left( \frac{1}{2} - \frac{1}{4} - \frac{1}{5}
	 \right) = 
	 \frac{\delta\,r}{20\,(2+\delta)} \, .
\end{align}
 This implies $U$ is increasing on this interval and that there exists an $R>0$ and $c>0$ such that  $ U(r)\geq c\,r^2$
 for $r>R$.
Thus,  by Corollary \ref{c:driftUmont}, if $ \frac{r^2}{2}-r>U$ for $r > R$, then
 \begin{align}
 (\log U)'\geq \frac{2-n}{r}+1-\frac{\lambda}{c\,r}\,  .
 \end{align}
This forces $U$ to grow exponentially to the top of this range, eventually giving the claim.   
 \end{proof}

\begin{proof}
(of Theorem \ref{t:main}).   Since $(\log I)' \leq 2 \, U/r$ for $r > \sqrt{2n}$ by Lemma \ref{l:Idiff}, 
the growth bound \eqr{e:IU08} will follow from the  bound \eqr{e:0p9} on $U$.
 We first show  for any $\delta > 0$ that
\begin{align}
U(r) \leq 2\,\lambda+\delta\label{e:weakerbound}
\end{align}
for all $r$ sufficiently large. We will argue by contradiction, 
so  suppose that  \eqr{e:weakerbound} fails for some $r$ sufficiently large.  
Theorem \ref{c:Uquad} gives that $U \geq \frac{r^2}{2} - r  $ for all sufficiently large $r$.  It follows that $K(r) = D(r) - 4\, \lambda \, I(r)$ is positive for all large $r$.  At a regular value $r > 2 \, \sqrt{n}$, Proposition \ref{p:driftUmont} and Lemma  \ref{l:Idiff} give
\begin{align}	\label{e:e313}
	r\, K'   &\geq \left( 2-n  + \frac{r^2}{2}  + U   - 8 \, \lambda \right) \, D - \lambda \, r^2 \, I
	 + \left[ 8\, (1 - \frac{2\,n}{r^2}) - 4 \right] \, \lambda \, r^{1-n}  \int_{b = r} \frac{S\,|u|^2}{ |\nabla b|} \notag \\
	&\geq  \left( 2-n  + r^2 - r  - 8 \, \lambda \right) \, D - \lambda \, r^2 \, I \\
	&\geq  \left( 2-n  + r^2 - r  - 8 \, \lambda \right) \, K +  4\, \lambda \, \left( 2-n  + \frac{3\,r^2}{4} - r  - 8 \, \lambda \right) \, I
	   \, . \notag 
\end{align}
 Thus, for $r$ large, we have $(\log K)' \geq \frac{3}{4} \, r$.  Integrating this gives 
 for $t>s>R$ 
\begin{align}	\label{e:inttocont}
	D(t)\geq K(t) \geq K(s) \,  \e^{ \frac{ 3\, (t^2-s^2)}{8}} \, .
\end{align}
This implies that 
\begin{align}
2\, \int_{b\leq t}(|\nabla u|^2 + \langle \cL \, u , u \rangle)\,\e^{-\frac{|x|^2}{4}} = \int_{b\leq t} \cL\, |u|^2
\,\e^{-\frac{|x|^2}{4}}&=2\, \e^{-\frac{t^2}{4}}\,t^{n-2}\,D(t) \to \infty  {\text{ as }} t \to \infty \, . \notag
\end{align}
This is a contradiction since $\cL \, u \in L^2$ and $u \in W^{1,2}$ by Lemma \ref{l:W12},  so   \eqr{e:weakerbound} holds.

We turn to the sharper bound \eqr{e:0p9}; we can assume that $\lambda > 0$ since otherwise $u$ is parallel since $u, \cL \, u \in L^2$.
The proof is by contradiction, so suppose that 
  $r \geq R$ satisfies
\begin{align}
2\,\lambda+\delta\geq U(r)\geq 2\,\lambda\,\left(1+\frac{\mu}{r^2}\right)\, , \label{e:therange}
\end{align}
where $\mu \in \RR$ will be chosen below.
 At any $r$ satisfying \eqr{e:therange}, we have 
\begin{align}
\frac{r^2}{2}-\frac{\lambda\,r^2}{U}=\frac{r^2}{2}\,\left(1-\frac{2\,\lambda}{U}\right)\geq \frac{\lambda\,\mu}{U} \geq 
 \frac{\lambda\,\mu}{2\,\lambda+\delta} \,  .
\end{align}
Together with \eqr{e:logU}, this gives at regular values that 
\begin{align}   \label{e:improvedlogU}
	r\, \left(\log U\right)' (r)
&\geq 2-n+\frac{\lambda \, \mu}{2\, \lambda + \delta} -2\,\lambda-\delta+
	\left( \frac{\lambda\,\mu}{2\,\lambda+\delta} -n\right)\, \frac{4\,r^{-1-n}}{I}\int_{b=r}\frac{S\,|u|^2}{|\nabla b|}\, .
\end{align}
Assuming that $\mu>\left( 2 + \frac{\delta}{\lambda} \right)\, n$ so the last term is nonnegative, we have
\begin{align}
r\, \left(\log U\right)' (r)&\geq \frac{ \lambda \, \mu - (2\, \lambda + \delta)^2 - (n-2) (2\, \lambda  + \delta)}{2\, \lambda + \delta} \, .
\end{align}
If   $\mu > 4 \, \lambda +2n - 4$, then this is strictly positive for $\delta > 0$ sufficiently small, forcing $U$ to 
  grow out of the range \eqr{e:therange}, giving the desired contradiction.
\end{proof}

\subsection{Examples}

We will next 
  consider  examples which show that Theorem \ref{t:main} is surprisingly sharp.  Not only is the threshold $2\lambda$ sharp, but even the next order term is sharp.  
If $u= b^2 - 2n$, then  $\cL \, u = - u$, so that $\lambda = 1$, and \eqr{e:firstclaim} gives
  \begin{align}
   	D(r) &= \frac{r^{2-n}}{2} \, \int_{b = r} \langle \nabla (b^2 -2n)^2 , \frac{\nabla b}{|\nabla b|} \rangle  =
	2\, r^{3-n} \, (r^2 - 2\,n) \,  \int_{b = r}     |\nabla b|  = \frac{2\, r^2 \, I(r)}{r^2-2n}  \, . 
\end{align}
Therefore, we see that the frequency $U= \frac{D}{I}$ satisfies
\begin{align}
	U(r) &=   \frac{2\, r^2}{r^2-2n} = 2\,  \left(1 + \frac{2\,n}{r^2} + O(r^{-4}) \right) = 2\, \lambda \, \left(1 + \frac{ 4\, \lambda + 2\, n -4}{r^2} + O(r^{-4}) \right) \, .
\end{align}
Next, let $M= \RR$, 
   $f = \frac{x^2}{4}$, and $\cL$ be the Ornstein-Uhlenbeck operator.
The degree $m$ Hermite polynomial has $\lambda = \frac{m}{2}$ and is given by
	$x^m - m \, (m-1) \, x^{m-2} + O(x^{m-4})$, so that
 \begin{align}
 	I(r) = 2 \, \left( r^{2\,m} - 2 \, m \, (m-1) \, r^{2\,(m-1)} + O(r^{2\,(m-2)}) \right) \, .
 \end{align}
 It follows that
 \begin{align}
 	2\, U(r) = \frac{r\, I'}{I} = 2\,m \, \frac{ r^{2\,m} - 2 \, (m-1)^2 \, r^{ 2\,(m-1)} + O(r^{2\,(m-2)})}{r^{2\,m} - 2 \, m \, (m-1) \, r^{2\,(m-1)} + O(r^{2\,(m-2)})} \, .
 \end{align}
 Thus, we have
	$U(r) = m \, \left( 1 + 2\, (m-1) \, r^{-2}
	+ O(r^{-4}) \right) = 2\, \lambda \,   \left( 1 +  (4\, \lambda - 2) \, r^{-2}
	+ O(r^{-4}) \right) $.

\section{Poisson equation}

Suppose that $u$ satisfies $\langle \cL \, u , u \rangle  \geq - \lambda \, |u|^2 - \psi$,  
where $\lambda \geq 0$ is a constant and $\psi \geq 0$ is a function.    By
  Lemma \ref{l:coareaapp},  $J$ from  \eqr{e:J} is absolutely continuous and $J'$ is given a.e. by
 \begin{align}   \label{e:J'}
 	J' = r^{2-n} \, \int_{b=r} \frac{\psi}{|\nabla b|} \, .
 \end{align}
  We will use the following immediate analog of  Proposition \ref{p:driftUmont} (with the additional term  in $D'$ (cf. \eqr{e:thirdclaim}), resulting in  $J'$ terms in \eqr{e:thederivofD2A}, \eqr{e:logUA}).

   \begin{Lem}	\label{l:2G}
   If $r$ is a regular value of $b$ and $D(r) , I(r) > 0$, then 
  \begin{align}
	r\,(\log D)' &\geq 2-n  + \frac{r^2}{2}  + U -\frac{\lambda\,r^2}{U} 
	 -\frac{4\,\lambda}{U\,I} \, r^{1-n}\, \int_{b = r}
	\frac{S\,|u|^2}{|\nabla b|} - \frac{r}{D} \, J'
	\, , \label{e:thederivofD2A}\\
	r\, \left(\log U\right)' &\geq    2-n+\frac{r^2}{2}-U-\frac{\lambda\,r^2}{U} +
	\left(1-\frac{2\,\lambda}{U}- \frac{2\,n}{r^2}\right) \, \frac{2\,r^{1-n}}{I}\int_{b=r}\frac{S\,|u|^2}{|\nabla b|}  - \frac{r}{D} \,J'
\, .\label{e:logUA}
 \end{align}
   \end{Lem}

 \begin{Lem}  \label{l:Knu}
 Given $\delta \in (0,2)$, set $K =D-(2\,\lambda + \delta/2) \,I$.  There exists $r_0(\lambda , \delta , n)$, so that if $r \geq r_0$ is a regular value with $K(r) > 0$, then
  \begin{align}
 r\,K '\geq \frac{2\, \lambda \,r^2}{4\, \lambda+\delta}\,K+\left[U+2-n+\frac{\delta\,r^2}{2\,(4\lambda + \delta)}-(4\, \lambda +\delta) \right]\,D-r\,J'\,  .
 \end{align}
 \end{Lem}
 
 \begin{proof}
 By \eqr{e:thederivofD2A} and \eqr{e:secondclaim}, we have
  \begin{align}
  	r\, D' &\geq \left( 2-n  + \frac{r^2}{2}  + U\right) \, D - \lambda\,r^2 \, I 
	 -4\,\lambda \, r^{1-n}\, \int_{b = r}
	\frac{S\,|u|^2}{|\nabla b|} - r\, J' \, ,\\
 r\, (2\,\lambda + \delta/2)  \, I' &= (4\, \lambda + \delta)\, D + \left( \frac{2\,n}{r^2}-1 \right)\, (4\, \lambda + \delta) \, r^{1-n}  \int_{b = r} \frac{S\,|u|^2}{ |\nabla b|}  \, .
 \end{align}
 Since $S\geq 0$ and  $\left[(4\, \lambda + \delta)(1-2\,n\,r^{-2})-4\, \lambda \right]\geq 0$ for $r\geq r_0(\lambda , \delta , n)$,  it follows that 
 \begin{align}
 r\,K'& \geq \left[\left( 2-n  + \frac{r^2}{2}  + U\right)  -(4\, \lambda + \delta)\,\right]\,D-\lambda\,r^2 \, I-r\,J'\,  .
 \end{align}
Since  $[D-2\,\lambda\,I]=\frac{4\, \lambda}{4\, \lambda +\delta}\,K+\frac{\delta}{4\, \lambda +\delta}\,D$, this gives the claim.  
  \end{proof}
  
     \begin{proof}[Proof of Theorem \ref{t:main2A}]
      Set $J_0 = \sup \, J$.
 We will show that 
 \begin{align}	\label{e:showbycdA}
 	K(r)  \leq  10 \, J_0 {\text{ for all }} r>R(\lambda,\delta,n) \, .
 \end{align}
Once we have \eqr{e:showbycdA}, we use   \eqr{e:secondclaim} to get that  
\begin{align}
	r\, I' &\leq 2\,D  \leq (4 \, \lambda + \delta) \, I +  20\,J_0  \, .
\end{align}
Equivalently, 
$\left(r^{-(4 \, \lambda +  \delta)}\,I\right)'\leq 20\,r^{-(4 \, \lambda +  \delta)-1}\,J_0$.  Integrating this gives
\eqr{e:main2A}.  

We will prove \eqr{e:showbycdA} by contradiction, so suppose instead that $K(r_0) > 10 \, J_0$ for some large  $r_0$.  At any regular value $r$ with $K(r) > 0$, we have $D(r) > 0$,  thus, also
$I(r) > 0$ by Lemma \ref{l:IBDR} and $U(r) > 2\, \lambda+ \delta/2 > 0$.  Lemma \ref{l:Knu} then implies that if $r$ is large enough and $K > 0$, then $K' \geq - J'$.  Integrating this from $r_0$ gives that $K(r) \geq 9 \, J_0$ for all $r\geq r_0$ and, thus, also that $D, I > 0$ and 
$U > ( 2\, \lambda + \delta /2) > 0$.   In particular, \eqr{e:logUA}
gives
\begin{align}	\label{e:UplusJa}
	\left(\log U\right)' &\geq    \frac{2-n- U}{r} +\frac{r}{2}-\frac{\lambda\,r}{U}    - \frac{J'}{D} \geq
	 \frac{2-n- U}{r} +\frac{r}{2}-\frac{\lambda\,r}{U}    - \frac{J'}{9 \, J_0}
\,  .
\end{align}
  Suppose first $U(r) < \frac{\delta \, r^2}{4(4\, \lambda + \delta)}$ for every larger $r$, then \eqr{e:UplusJa} would give
\begin{align}	\label{e:UplusJb}
	\left(\log U\right)' &\geq 	 \frac{2-n}{r}-  \frac{\delta \, r}{4(4\, \lambda + \delta)} +\frac{r}{2}-\frac{2\,\lambda\,r}{4\, \lambda+ \delta}    - \frac{J'}{9 \, J_0} =  	 \frac{2-n}{r}-   \frac{J'}{9 \, J_0} +  \frac{\delta \, r}{4(4\, \lambda + \delta)}
\,  .
\end{align}
Integrating this contradicts the upper bound on $U$, so we conclude that   there is a large $r$ where $U \geq \frac{\delta \, r^2}{4(4\, \lambda + \delta)}$.   Next, at any large $r$ where 
	$\frac{\delta \, r^2}{8(4\, \lambda + \delta)}  \leq U(r) \leq \frac{r^2}{2} - r$, 
then \eqr{e:UplusJa} gives
\begin{align}	\label{e:UplusJc}
	\left(\log U\right)' &\geq    1 +
	 \frac{2-n}{r}  -\frac{8\, \lambda \, (4\, \lambda + \delta)}{\delta \, r}    - \frac{J'}{9 \, J_0}
\,  ,
\end{align}
 forcing $U$ to grow exponentially and, thus, eventually overtake the quadratic upper bound. Thus, we get $R_1$ large so that for all $ r \geq R_1$ we have $U > \frac{r^2}{2} -r  - \frac{1}{9}$  (the last term comes from integrating $\frac{J'}{9\, J_0}$).
   Using this lower bound for $U$ in Lemma \ref{l:Knu} gives
\begin{align}
( K+ J)' &\geq \frac{2\, \lambda \,r}{4\, \lambda +\delta}\,K+\left[ \left( \frac{r^2}{2} - r - \frac{1}{9} \right) +2-n+\frac{\delta\,r^2}{2\,(4\lambda+\delta)}-(4\, \lambda+ \delta) \right]\,\frac{K}{r} \\
&= \left(r - 1 + \frac{2-n-1/9 - (4\, \lambda + \delta)}{r} \right) \, K 
\geq \frac{8\, r}{9} \, K \geq \frac{4\, r}{5} \, (K+J) \, ,	\notag
 \end{align}
 where the last inequality used $K+J \leq K + J_0 \leq \frac{10}{9} \, K$.  Integrating gives that $K+J$ grows at least like $\e^{ \frac{2\, r^2}{5}}$.  This contradicts that $u \in W^{1,2}, \cL \, u \in L^2$ as in the proof of Theorem \ref{t:main}.
 \end{proof}

 We will also prove an effective growth bound similar in spirit to Hadamard's three circles theorem, \cite{Li}, \cite{N}.  Roughly, this shows that if $u$ is very small on a scale $r_1$ and bounded at larger scale $R$, then $u$ stays small out to  scale $R-1$.
  
 \begin{Pro}	\label{p:effective}
 Given  $ \lambda >0$ and $\delta \in (0,2\, \lambda)$, there exists $r_0$ so that if $r_0 \leq r_1 < R$, $u$ satisfies \eqr{e:herelambda} on $\{ r_1 \leq b \leq R \}$ and  $D(R) \leq \e^{ \frac{2\, R - 1}{6}} \, I(r_1)$, then for all $r \in [r_1 , R-1]$ 
\begin{align}
	I(r) \leq \left( \frac{r}{r_1} \right)^{ 4\, \lambda +2\, \delta} \, \left[1+\frac{1}{(2\, \lambda+ \delta)}  \right]\, I(r_1) \, .
\end{align}
 \end{Pro}

 \begin{proof}
 By Lemma \ref{l:Knu} with $J = 0$,  if $r\geq r_0=r_0(\lambda,\delta,n)$ and $K(r) > 0$, then $K' \geq \frac{r}{3} \, K$ and, thus, 
 $\e^{-\frac{r^2}{6}}\,K(r)$ is monotone non-decreasing.  If    $r\in [r_1,R-1]$ with $K(r) > I(r_1)$, then   $D(r) > K(r) > 0$ and, thus, also
$I(r) > 0$ by Lemma \ref{l:IBDR}.  Moreover, 
\begin{align}
	 D(R) > K(R) \geq \e^{  \frac{R^2 - r^2}{6}} \, K(r) \geq \e^{ \frac{R^2 - r^2}{6}} \, I(r_1) \geq  \e^{ \frac{2\, R - 1}{6}} \,  I(r_1) \, .
\end{align}
This contradicts $D(R) \leq \e^{ \frac{2\, R - 1}{6}} \, I(r_1)$, so 
  $K(r) \leq I(r_1)$     for all $ r \in (r_1 , R - 1)$ and, thus,
   \begin{align}
   	D(r) = K(r) + (2\, \lambda +\delta)  \, I(r) \leq I(r_1)+ (2\, \lambda+\delta )  \, I(r) \, .
\end{align}
Combining this with the bound on $I'$ from  Lemma \ref{l:IBDR} gives 
\begin{align}
	\left( r^{-(4\, \lambda +2\, \delta) } \, I(r) \right)' \leq - (4\, \lambda +2\, \delta)\, r^{-(4\, \lambda+2\, \delta)-1} \, I + 2 \,  r^{-(4\, \lambda+2\, \delta)-1} \, D
	\leq 2\, r^{-(4\lambda +2\, \delta)-1} \, I (r_1)\, .  \notag
\end{align}
Integrating from $r_1$ to $r \leq R -1$ gives the claim.  
 \end{proof}

\end{document}